\documentclass[preprint,12pt]{elsarticle}

\usepackage{savesym}
\savesymbol{bibsection}

\usepackage{graphics}
\usepackage{cancel}
\usepackage{tikz}
\usetikzlibrary{matrix,arrows}
\usepackage{amsmath}
\usepackage{amsthm}
\usepackage{amssymb}
\usepackage{amsrefs}
\usepackage{verbatim}
\usepackage{setspace}
\usepackage{hyperref}
\usepackage[normalem]{ulem}

\newcommand{\V}{\mathcal{V}}
\newcommand{\K}{\mathrm{K}}

\newcommand{\into}{\hookrightarrow}
\newcommand{\onto}{\twoheadrightarrow}
\newcommand{\op}{\mathrm{op}}
\renewcommand{\hat}{\widehat}
\newcommand{\gen}[1]{\langle #1 \rangle}

\theoremstyle{definition}
\newtheorem{theorem}{Theorem}
\newtheorem{definition}[theorem]{Definition}
\newtheorem{lemma}[theorem]{Lemma}

\newtheorem{corollary}[theorem]{Corollary}
\newtheorem{proposition}[theorem]{Proposition}
\newtheorem{example}[theorem]{Example}
\newtheorem{remark}[theorem]{Remark}

\newtheorem*{outline}{Outline of the paper}

\numberwithin{theorem}{section}

\newcommand*\cocolon{%
\nobreak
\mskip6mu plus1mu
\mathpunct{}%
\nonscript
\mkern-\thinmuskip
{:}%
\mskip2mu
\relax
}

\newcommand{\eq}{\approx}

\newcommand{\m}[1]{{\bf {#1} }}

\renewcommand{\a}{\ensuremath{\alpha}}
\renewcommand{\b}{\ensuremath{\beta}}
\newcommand{\ga}{\ensuremath{\gamma}}
\newcommand{\si}{\ensuremath{\sigma}}
\newcommand{\de}{\ensuremath{\delta}}

\DeclareMathOperator{\Con}{\mathrm{Con}}
\DeclareMathOperator{\KCon}{\mathrm{KCon}}
\newcommand{\cg}[1]{{\rm Cg}_{_{#1}}}

\newcommand{\cls}[1]{\mathcal{#1}}
\newcommand{\mdl}[1]{\models_{#1}}

\newcommand{\lgc}[1]{\ensuremath{{\sf #1}}}
\newcommand{\der}[1]{\ensuremath{\vdash_\mathbf{#1}}}
\newcommand{\De}{\Delta}
\newcommand{\Ga}{\Gamma}
\newcommand{\Si}{\Sigma}

\newcommand{\ep}{\varepsilon}

\newcommand{\F}{\m{F}}
\newcommand{\Tm}{\m{T}}
\newcommand{\Tmc}{\mathrm{T}}
\newcommand{\lang}{\mathcal{L}}
\newcommand{\x}{\overline{x}}
\newcommand{\y}{\overline{y}}
\newcommand{\z}{\overline{z}}
\newcommand{\w}{\overline{w}}
\newcommand{\uvar}{\overline{u}}

\journal{Annals of Pure and Applied Logic} 

\begin{document}

\begin{frontmatter}

\title{Uniform Interpolation and Compact Congruences}

\author{Samuel J. van Gool\fnref{thankssam}}
\address{Department of Mathematics, City College of New York, U.S.A.}
\address{ILLC, University of Amsterdam, The Netherlands}
\ead{samvangool@me.com}
\fntext[thankssam]{Supported by the European Union's Horizon 2020 research and innovation programme under the Marie Sklodowska-Curie grant agreement 655941.}

\author{George Metcalfe\corref{cor}\fnref{thanks}}
\address{Mathematical Institute, University of Bern, Switzerland}
\ead{george.metcalfe@math.unibe.ch}
\cortext[cor]{Corresponding author}
\fntext[thanks]{Supported by Swiss National Science Foundation (SNF) grant 200021{\_}146748.}	

\author{Constantine Tsinakis}
\address{Department of Mathematics, Vanderbilt University, U.S.A.}
\ead{constantine.tsinakis@vanderbilt.edu}


\begin{abstract}
Uniform interpolation properties are defined for equational consequence in a variety of algebras and related to properties of compact congruences on first the free and then the finitely presented algebras of the variety. It is also shown, following related results of Ghilardi and Zawadowski, that a combination of these properties provides a sufficient condition for the first-order theory of the variety to admit a model completion.
\end{abstract}


\end{frontmatter}


\section{Introduction}

The following remarkable feature of intuitionistic propositional logic $\lgc{IPC}$ was established by A.~M.~Pitts in~\cite{Pit92}. Given any formula $\a(\bar{x},\bar{y})$ of the logic, there exist formulas $\a^L(\bar{y})$ and $\a^R(\bar{y})$, {\em left} and {\em right uniform interpolants} of $\a$ with respect to $\bar{x}$, respectively,
such that for any formula $\b(\bar{y},\bar{z})$,
\[
\der{IPC} \b \to \a \iff \ \der{IPC} \b \to \a^L
\quad \text{and} \quad
\der{IPC} \a \to \b \iff\ \der{IPC} \a^R \to \b.
\]
Each of the seven intermediate propositional logics admitting the usual Craig interpolation property also admits uniform interpolation; however, there are modal logics, such as $\lgc{S4}$, that admit Craig interpolation but not uniform interpolation (see~\cite{GZ02} for details and references).

The main aim of this paper is to study uniform interpolation in the setting of universal algebra. Instead of a propositional logic, we consider a variety of algebras: a class of algebraic structures of the same signature that is defined by equations (equivalently, closed under homomorphic images, subalgebras, and direct products). Although there may exist faithful translations between the variety and a propositional logic (e.g., Boolean algebras and classical logic, Heyting algebras and intuitionistic logic, MV-algebras and {\L}ukasiewicz logic), this need not be the case. In particular, the signature may not contain a `suitable' implication connective. We therefore focus here on consequences in the variety between a set of equations on the left and a single equation on the right. The most natural interpolation property in this setting is {\em deductive interpolation} (studied in~\cites{Jon65,Pig72,Mak77,Wro84a,Cze85,Ono86,Cze07,MMT14}), which coincides with Craig interpolation only in the presence of a suitable deduction theorem. Indeed, a variety with the congruence extension property admits deductive interpolation if, and only if, it has the amalgamation property (see~\cite{MMT14} for proofs and references). In this paper, we define left and right uniform interpolation properties for a variety and obtain corresponding algebraic characterizations. We also obtain algebraic characterizations for uniform versions of the Maehara interpolation property studied in~\cites{Wro84a,Cze85,MMT14}.

The starting point for our study of uniform interpolation is the extensive work on interpolation and model completions by Ghilardi and Zawadowski, collected in the monograph~\cite{GZ02}. These authors studied category-theoretic properties of varieties that correspond to propositional logics with (left and right) uniform interpolation, establishing these properties for certain varieties of Heyting and modal algebras. Our motivations here for supplementing this category-theoretic view of uniform interpolation with a universal algebraic perspective are two-fold. First, we define left and right uniform interpolation as specific properties of a variety, whereas in~\cite{GZ02} these arise as combinations of other properties. This allows us to identify new examples of varieties with and without uniform interpolation, not restricted to intermediate and modal logics, including groups, MV-algebras, implicative semilattices, Sugihara monoids, and bounded lattices. Second, our algebraic perspective exhibits connections between uniform interpolation and known properties in equational logic and universal algebra. In particular, we obtain uniform versions of Maehara interpolation and a better understanding of the connection between uniform interpolation and amalgamation.

A further logical motivation for studying uniform interpolation lies in its relationship to the notion of a model completion, which originated in the groundbreaking work on model-theoretic algebra of A.~Robinson \cite{Rob63}. A model completion of a first-order theory axiomatizes the class of algebras in which `all potentially solvable equations possess solutions' and always has quantifier elimination. The prototypical example is the theory of fields, whose model completion is the theory of algebraically closed fields. Ghilardi and Zawadowski showed in~\cite{GZ02} that category-theoretic properties of varieties of algebras for intermediate and modal logics with uniform interpolation imply the existence of a model completion for the first-order theory of the variety. Building on their work, we relate the algebraic uniform interpolation properties introduced in this paper to the existence of a model completion. This approach is also related to the use in~\cites{Mar12,MM12} of model-theoretic methods (in particular, quantifier elimination) to establish the amalgamation property for certain varieties of semilinear commutative residuated lattices. 

The key notions needed for this algebraic view of uniform interpolation turn out to be compact (i.e., finitely generated) congruences and pairs of adjoint maps between them. The importance of compact congruences was already implicit in \cite{GZ02}, where they appear as `regular monomorphisms in the opposite of a category of finitely presented algebras'. A notable conceptual contribution of our work is that we view these notions as central to the algebraic study of uniform interpolation. Compact congruences are needed in order to obtain a stable theory of uniform interpolation when venturing outside the realm of Heyting and modal algebras considered in~\cite{GZ02}, where compact congruences may be harmlessly identified with certain elements of algebras (see also Remark~\ref{rem:heyting} below). In other settings, such as that of groups, more care is required, and the semilattice of compact congruences is the appropriate setting for studying uniform interpolation.

\begin{outline}
After introducing the necessary background and notation in Section~\ref{sec:prelim}, we establish the main theoretical results for right and left uniform interpolation in Sections~\ref{sec:right}~and~\ref{sec:left}, respectively. In particular, Theorems~\ref{thm:rudipalgebraic} and~\ref{thm:ludipalgebraic} characterize varieties admitting an algebraic form of right and left uniform interpolation. In Section~\ref{sec:modelcompletions}, we recast the results of~\cite{GZ02} in this algebraic light, using the characterizations of the previous sections to give an algebraic proof of the existence of a model completion in the presence of right and left uniform interpolation. Throughout the paper, we consider a range of examples of varieties for which uniform interpolation properties hold, or fail to hold.
\end{outline}


\section{Congruence and consequence}\label{sec:prelim}

In this section, we recall notions and basic results from universal algebra required for the paper, focusing in particular on (compact) congruences of algebras and how they relate to equational consequence in a variety. In Subsection~\ref{subsec:congprelim}, we consider join-semilattices of compact congruences, and liftings of homomorphisms to these semilattices. In Subsection~\ref{subsec:consprelim}, we describe the relationship between equational consequence and congruences on free algebras of a variety. Finally, in Subsection~\ref{subsec:dedinterp}, we recall the notion of deductive interpolation and related algebraic properties such as amalgamation. We assume familiarity with basic definitions from universal algebra, as provided in, e.g.,~\cite{BS81}*{Ch. I--II}. 


\subsection{Congruence lattices and lifting homomorphisms}\label{subsec:congprelim}

An element $k$ in a complete lattice $\m{L}$ is {\it compact} if, whenever $k \leq \bigvee S$ for some $S \subseteq L$, there exists a finite $F \subseteq S$ such that $k \leq \bigvee F$. The set of compact elements of $\m{L}$ forms a join-subsemilattice of $\m{L}$, denoted by $\K\m{L}$. An {\it algebraic} lattice is a complete lattice in which every element is a join of compact elements (see, e.g.,~\cites{CD73, DP02, Gierz03} for further details). 

For any algebra $\m A$, the set of congruences on $\m A$ is denoted by $\Con \m A$. The intersection of a set of congruences on $\m A$ is again a congruence on $\m A$, so $\Con \m{A}$ is a complete lattice under the inclusion order.\footnote{We adhere here to the usual convention in universal algebra that congruences are ordered by set-theoretic inclusion; hence, $\theta \leq \psi$ if, and only if, the natural map from $\m A/{\theta}$ to $\m A/{\psi}$ is well-defined. We warn the reader that this convention sometimes makes an order-reversal necessary when comparing with other settings, cf., e.g., Remarks~\ref{rem:catthry}~and~\ref{rem:heyting}.} In fact, $\Con \m A$ is always an algebraic lattice~\cite{BS81}*{Theorem~5.5}. For any $S \subseteq A \times A$, we denote by $\cg{\m A} S$ the {\it congruence on $\m A$ generated by $S$}. Recall that the compact elements of $\Con \m A$ are exactly the {\it finitely generated} congruences; we denote the set of these compact congruences on $\m A$ by $\KCon \m{A}$. Note that, although $\KCon \m A$ is always a join-subsemilattice of $\Con \m A$, meets in $\KCon \m A$ need not exist in general. 

A pair of maps $f \colon \m P \leftrightarrows \m Q \cocolon g$ between partially ordered sets $\m{P},\m{Q}$ is called an \emph{adjunction} or \emph{adjoint pair} if for all $a \in P$, $b \in Q$, it holds that $f(a) \leq b$ if, and only if, $a \leq g(b)$. The map $f$ is called \emph{left adjoint} to $g$, and $g$ is called \emph{right adjoint} to $f$. Note that if $f, g$ is an adjoint pair, then they are both monotone; in fact, $f$ preserves arbitrary existing joins in $\m P$ and $g$ preserves arbitrary existing meets in $\m Q$. 
In this case, the composite map $g \circ f$ is a \emph{closure operator} on $P$. We refer to, e.g.,~\cite{DP02}*{Ch.~7} or~\cite{Gierz03}*{Sec.~O-3} for more background on adjunctions and closure operators.

In the following definition, central to this paper, we lift a homomorphism between algebras to an adjoint pair between their congruence lattices. 

\begin{definition}\label{dfn:lifting}
Let $h \colon \m A \to \m B$ be a homomorphism. The \emph{adjoint lifting} of $h$ to the congruence lattices of $\m{A}$ and $\m{B}$ is the pair of maps
\begin{center}
$h^* \colon \Con \m{A} \leftrightarrows \Con \m{B} \cocolon h^{-1}$,
\begin{align*}
h^*(\psi) &:= \cg{\m B}\{(h(a),h(a')) : (a,a') \in \psi\}, \\
h^{-1}(\theta) &:= \{(a,a') \in A^2 : (h(a),h(a')) \in \theta\} = \ker h(-)/{\theta}.
\end{align*}
\end{center}
We call $h^*$ the {\it direct image lifting} and $h^{-1}$ the {\it inverse image lifting} of $h$. We denote the composite map $h^{-1} \circ h^*$ by $c_h \colon \Con \m A \to \Con \m A$. 
\end{definition}

For any homomorphism $h \colon \m A \to \m B$, $h^*$ is left adjoint to $h^{-1}$ and hence the map $c_h$ is a closure operator on the complete lattice $\Con \m A$. Moreover, $h^*$ restricts to a map $\KCon \m A \to \KCon \m B$, which we call the \emph{compact lifting} of $h$. Note, however, that $h^{-1}$ need not preserve compact congruences in general. (Consider, for example, any non compact congruence $\theta$ on an algebra $\m{A}$ and a homomorphism $h \colon \m{A} \to \m{A / \theta}$ such that $\theta = h^{-1}(\De_\m{A})$.)

Regarding the existence of a right adjoint to the compact lifting of a homomorphism, we have the following general fact.

\begin{proposition}\label{prop:adjcompact}
Let $f \colon \m{L} \leftrightarrows \m{M} \cocolon g$ be an adjunction between algebraic lattices such that $f$
preserves compact elements. The restriction $f|_{\K\m{L}}$ has a right adjoint if, and only if, $g$ preserves
compact elements, and in this case, the right adjoint of $f|_{\K\m{L}}$ is $g|_{\K\m{M}}$.
\end{proposition}
\begin{proof}
If $g$ preserves compact elements, then $g|_{\K\m{M}} \colon \K\m{M} \to \K\m{L}$ is clearly right adjoint to $f|_{\K\m{L}}$. Conversely, suppose that $f|_{\K\m{L}}$ has a right adjoint $g'$. Let $m \in \K\m{M}$. Then $fg'(m) \leq m$, so $g'(m) \leq g(m)$, since $g$ is right adjoint to $f$. Moreover, for any $l \in \K\m{L}$ such that $l \leq g(m)$, we have $f(l) \leq m$, and hence $l \leq g'(m)$, since $g'$ is right adjoint to $f|_{\K\m{M}}$. As $g(m)$ is the join of compact elements below it, $g(m) \leq g'(m)$. So $g(m) = g'(m)$ is compact. \end{proof}

\noindent
In particular, for any homomorphism $h \colon \m A \to \m B$ between algebras, $h^*|_{\KCon \m{A}}$ has a right adjoint if, and only if, $h^{-1}$ preserves compact congruences.

\begin{remark}
\label{rem:catthry}
Let $\V$ be a variety. The congruence lattice of an algebra $\m A$ in $\V$ is isomorphic to the opposite of the lattice $\mathrm{rSub}_{\V^\op}(\m A)$ of {\it regular subobjects} of $\m A$, regarded as an object in the category $\V^\op$. Under this anti-isomorphism, the adjoint lifting $(h^*,h^{-1})$ of $h$ corresponds to the adjunction, usually denoted by $(\exists_h,h^*)$ in categorical logic. Moreover, if $\m A$ is finitely presented, then the join-semilattice of compact congruences is isomorphic to the opposite of the lattice $\mathrm{rSub}_{\V_{\mathrm{fp}}^\op}(\m A)$ of regular subobjects of $\m A$ in the full subcategory $\V_{\mathrm{fp}}^\op$ of $\V^\op$ whose objects are the finitely presented algebras. We refer to, e.g.,~\cite{GZ02}*{Ch.~2}, for more details.
\end{remark}

\begin{remark}
\label{rem:heyting}
The congruence lattice of any Heyting algebra $\m A$ is isomorphic to the filter lattice of $\m A$ (the isomorphism sends a congruence $\theta$ on $\m A$ to the $\theta$-equivalence class of the top element of $\m A$, see, e.g.,~\cite{BD74}). The compact filters are just the principal filters, and hence $\m A$ is isomorphic to the opposite of $\KCon \m{A}$. 
\end{remark}

We end this subsection by establishing a general lattice-theoretic fact about compact elements in intervals (Lemma~\ref{lem:compininterval}) and some properties of surjective homomorphisms (Proposition~\ref{prop:surjective}) that will be used in Section~\ref{sec:left}. These results belong to the folklore of the subject; we include proofs here for the sake of completeness.

\begin{lemma}\label{lem:compininterval}
Let $\m L$ be a complete lattice and $k \in \K \m L$. For any $a \in [k,\top]_{\m L}$, it holds that $a \in \K\m L$ if, and only if, $a \in \K([k,\top]_{\m L})$.
\end{lemma}
\begin{proof}
If $a \in \K \m L$, then certainly $a \in \K([k,\top]_{\m L})$. For the converse, suppose that $a \in \K([k,\top]_{\m L})$. Consider $S \subseteq L$ satisfying $a \leq \bigvee S$. Since $k \leq a$, we have $a = a \lor k \leq \bigvee_{s \in S} (s \lor k)$, and, by assumption, there exists a finite $F_1 \subseteq S$ such that $a \leq \bigvee_{s \in F_1} (s \lor k)$. Moreover, since $k \in \K \m L$ and $k \leq a \leq \bigvee S$, there exists a finite $F_2 \subseteq S$ such that $k \leq \bigvee F_2$. Combining the two inequations, we see that $a \leq \bigvee (F_1 \cup F_2)$, as required.
\end{proof}

\begin{proposition}[cf.~{\cite{MMT14}*{Lem.~1}}]\label{prop:surjective} 
Let $h \colon \m{A} \onto \m{B}$ be a surjective homomorphism. Then:
\begin{enumerate}
\item[(a)] For any $\psi \in \Con \m{A}$, $c_h(\psi) = \psi \lor \ker h$.
\item[(b)] The image of $c_h$ is the interval $[\ker h, \nabla_{\m A}]_{\Con \m A}$.
\item[(c)] If $\ker h$ is compact, then $h^{-1}$ preserves compact congruences.
\end{enumerate}
\end{proposition}

\begin{proof}
(a) Let us write $\theta$ for $\ker h$. The inclusion `$\supseteq$' is clear because $\theta = c_h(\De_{\m{A}})$ and $c_h$ is a closure operator. For the inclusion `$\subseteq$', note that, since $\psi \lor \theta$ is a congruence in the interval $[\theta,\nabla_{\m{A}}] \subseteq \mathrm{Con}(A)$, the Correspondence Theorem~\cite{BS81}*{Thm.~6.20} gives $c_h(\psi \lor \theta) = \psi \lor \theta$. Since $c_h$ is order-preserving, $c_h(\psi) \subseteq c_h(\psi \lor \theta)$, from which the required inclusion follows.

(b) This is a consequence of (a) and the fact that $\psi$ is in the image of $c_h$ if, and only if, $c_h(\psi) = \psi$.

(c) Since $h^{-1}$ is a lattice isomorphism from $\Con \m B$ to $[\ker h,\nabla_{\m A}]_{\Con \m A}$, it sends compact elements of $\Con \m B$ to compact elements of $[\ker h,\nabla_{\m A}]_{\Con \m A}$. By Lemma~\ref{lem:compininterval}, the compact elements of $[\ker h,\nabla_{\m A}]_{\Con \m A}$ are exactly those elements of $\KCon \m A$ that contain $\ker h$. 
\end{proof}


\subsection{Equational consequence and free algebras}\label{subsec:consprelim}

We briefly recall here some basic facts about free algebras and equational consequence, referring to, e.g.,~\cite{MMT14}*{Sec.~2} and~\cite{BS81}*{II.10} for more background. For convenience, let us assume for the rest of the paper that $\lang$ is a fixed algebraic signature with at least one constant symbol. For any (possibly infinite) set of variables $\x$, we denote by $\Tm(\x)$ the {\em $\lang$-term} {\em algebra over $\x$}, recalling that any assignment $f$ from $\x$ to the universe $A$ of an $\lang$-algebra $\m A$ lifts uniquely to a homomorphism $\hat{f} \colon \Tm(\x) \to \m A$. Elements $\a, \b$ of $\Tmc(\x)$ are called \emph{$\lang$-terms}; \emph{$\lang$-equations} are ordered pairs of elements $(\a,\b)$ of $\Tmc(\x)$, written also as $\a \eq \b$. We denote arbitrary $\lang$-equations by $\ep$ and sets of $\lang$-equations by $\Si,\De,\Ga,\Pi$, and write $\a(\x)$, $\ep(\x)$, or $\Si(\x)$ to denote that the variables of, respectively, an $\lang$-term $\a$, $\lang$-equation $\ep$, or set of $\lang$-equations $\Si$ are included in $\x$. We also adopt the convention that $\x,\y,\z$, etc. always denote disjoint sets of variables, and often write $\x,\y$ to denote their disjoint union.

Given a set of $\lang$-equations $\Si(\x)$, an $\lang$-algebra $\m A$, and an assignment $f \colon \x \to A$, we write $\m A, f \models \Si$ to denote that $\Si \subseteq \ker \hat{f}$. We also write $\m A \models \Si$ if $\m A, f \models \Si$ for all assignments $f \colon \x \to A$. Let $\cls{K}$ be any class of $\lang$-algebras. Then we define for a set of $\lang$-equations $\Si \cup \{\ep\}$ containing exactly the variables in the set $\x$,
\[
\begin{array}{rcc}
\Si \mdl{\cls{K}} \ep & \iff & \text{for every $\m{A} \in \cls{K}$ and assignment $f \colon \x \to A$,}\\[.1in]
& & \m A, f \models \Si \ \ \Rightarrow \ \ \m A, f \models \ep.
\end{array}
\]
Given a set of $\lang$-equations $\Si \cup \De$, we write $\Si \mdl{\cls{K}} \De$ if $\Si \mdl{\cls{K}} \ep$ for all $\ep \in \De$.

\begin{remark}
Informally, $\mdl{\cls{K}}$ may be called the `equational consequence relation' of $\cls{K}$. A genuine equational consequence relation (in the sense of Blok and Pigozzi~\cite{BP89}) is obtained only if we restrict the considered equations to a fixed set of variables. Note, however, that if $\cls{K}$ is a variety, then the consequence relation of $\cls{K}$ over a fixed countably infinite set of variables determines the consequence relation for any set of variables (see~\cite{MMT14}*{Sec.~2}). Also note that, in terms of first-order logic, if $\Si$ is finite, then $\Si \models_{\cls{K}} \ep$ is equivalent to: the universal sentence $\forall \x (\bigwedge \Sigma \rightarrow \ep)$ lies in the first-order theory of the class $\cls{K}$.
\end{remark}

Let us assume now and for the rest of the paper that $\V$ is a fixed variety of $\lang$-algebras and omit any further reference to the signature $\lang$. We denote the $\V$-free algebra on a set $\x$ of free generators by $\F(\x)$, constructed as the largest quotient of $\Tm(\x)$ that lies in $\V$, noting that both $\Tm(\x)$ and $\F(\x)$ exist even when $\x = \emptyset$, since $\lang$ contains a constant symbol by assumption. When clear from the context, we will write $\a$, $\ep$, or $\Si$ to denote also the element, pair of elements, or set of pairs of elements in $\F(\x)$ corresponding to a term $\a$, an equation $\ep$, or a set of equations $\Si$.

The fundamental connection between equational consequence and congruences on free algebras is given by the following lemma.

\begin{lemma}[{cf.~\cite{MMT14}*{Lem.~2}}]\label{lem:algcons} \
\begin{enumerate}
\item[\rm (a)]
For any set of equations $\Si(\x)$, algebra $\m{A}$, and assignment $f \colon \x \to A$,
\[
\m A, f \models \Si \iff \cg{\Tm(\x)} \Si \subseteq \ker \hat{f}.
\]
\item[\rm (b)]
For any sets of equations $\Si(\x), \De(\x)$,
\[
\Si \models_\V \De \iff \cg{\F(\x)} \De \subseteq \cg{\F(\x)} \Si.
\]
\end{enumerate}
\end{lemma}
\begin{proof}
(a) Clear from the definition and the fact that $\ker \hat{f}$ is a congruence on $\Tm(\x)$.

(b) Recall first that the variables $\y$ occurring in $\Si(\x), \De(\x)$ are included in (but may not equal) the set of variables $\x$ and observe that
\[
\begin{array}{rcc}
\Si \mdl{\V} \De & \iff & \text{for every $\m{A} \in \V$ and assignment 
$f \colon \x \to A$,}\\[.1in]
& & \m A, f \models \Si \ \ \Rightarrow \ \ \m A, f \models \De.
\end{array}
\]
The left-to-right direction follows by restricting an assignment $f \colon \x \to A$ to $g \colon \y \to A$; the right-to-left direction follows by extending any assignment $g \colon \y \to A$ arbitrarily to an assignment $f \colon \x \to A$. 

Suppose now that $\Si \models_\V \De$. Consider the natural assignment $f \colon \x \to F(\x)/\cg{\F(\x)} \Si$. Clearly $\F(\x)/\cg{\F(\x)} \Si, f \models \Si$, so $\F(\x)/\cg{\F(\x)} \Si, f \models \De$. Hence, by (a), $\De \subseteq \ker \hat{f} = \cg{\F(\x)} \Si$. So also $\cg{\F(\x)} \De \subseteq \cg{\F(\x)} \Si$. 

For the converse, suppose that $\cg{\F(\x)} \De \subseteq \cg{\F(\x)} \Si$. Let $h \colon \Tm(\x) \to \F(\x)$ be the natural homomorphism. Then
\[
h^*(\cg{\Tm(\x)} \De) = \cg{\F(\x)} \De \subseteq \cg{\F(\x)} \Si = h^*(\cg{\Tm(\x)} \Si).
\]
So, by Proposition~\ref{prop:surjective}(a),
\[
\cg{\Tm(\x)} \De \subseteq c_h(\cg{\Tm(\x)} \Si) = \cg{\Tm(\x)} \Si \lor \ker h.
\]
Now consider $\m{A} \in \V$ and an assignment $f \colon \x \to A$ satisfying $\m A, f \models \Si$ and hence, by (a), $\cg{\Tm(\x)} \Si \subseteq \ker \hat{f}$. Since $\ker h \subseteq \ker \hat{f}$, 
\[
\cg{\Tm(\x)} \De \subseteq \cg{\Tm(\x)} \Si \lor \ker h \subseteq \ker \hat{f}.
\]
That is, $\m A, f \models \De$. Hence $\Si \models_\V \De$.
\end{proof}

Note also that equational consequence in $\V$ is {\em finitary}; that is, for any set of equations $\Si \cup \{\ep\}$, whenever $\Si \models_\V \ep$, there exists a finite $\Si' \subseteq \Si$ such that $\Si' \models_\V \ep$. This follows from the previous lemma and the fact that, since any congruence lattice is algebraic, $\cg{\F(\x)} \Si = \bigcup \{\cg{\F(\x)} \Si' : \Si' \subseteq \Si, \Si' \mbox{ finite}\}$.


\subsection{Deductive interpolation}\label{subsec:dedinterp}

The following property of equational consequence may be viewed as a natural generalization of Craig interpolation for classical propositional logic to varieties of algebras and has been studied in depth by a number of authors (see, in particular,~\cites{Jon65,Pig72,Mak77,Wro84a,Cze85,Ono86,Cze07,MMT14}). 

\begin{definition}
$\V$ admits {\it deductive interpolation} if for any finite sets $\x,\y,\z$ and finite set of equations $\Si(\x,\y) \cup \{\ep(\y,\z)\}$ satisfying $\Si \mdl{\V} \ep$, there exists a finite set of equations $\Pi(\y)$ such that $\Si \mdl{\V} \Pi$ and $\Pi \mdl{\V} \ep$.
\end{definition}

Below we recall some useful reformulations of this property.

\begin{proposition}\label{prop:DIPBC}
The following are equivalent:
\begin{enumerate}
\item	$\V$ admits deductive interpolation.
\item For any finite sets $\x,\y$ and finite set of equations $\Si(\x,\y)$, there exists a set of equations $\Pi(\y)$ such that for any equation $\ep(\y,\z)$,
\[ 
\Si \mdl{\V} \ep \iff \Pi \mdl{\V} \ep.
\]
\item For any finite sets $\x,\y,\z$, the following diagram commutes:
\begin{equation*}
\begin{tikzpicture}[baseline=(current bounding box.center)]
\matrix (m) [matrix of math nodes, row sep=1.5em, column sep=1.5em]{
\Con \F(\x,\y) & & \Con \F(\y) \\
&	& \\
\Con \F(\x,\y,\z) & & \Con \F(\y,\z) \\
};
\path
(m-1-1) edge[->] node[above] {$i^{-1}$} (m-1-3)
(m-1-1) edge[->] node[left] {$j^*$} (m-3-1)
(m-3-1) edge[->] node[below] {$k^{-1}$} (m-3-3)
(m-1-3) edge[->] node[right] {$l^*$} (m-3-3);
\end{tikzpicture}
\end{equation*}
where $i$, $j$, $k$, and $l$ denote the inclusion maps between corresponding finitely generated free algebras.
\end{enumerate}
\end{proposition}
\begin{proof}
We consider first the diagram in (iii). 
For any $\theta \in \Con \F(\x,\y)$, let $\Si_\theta$ denote the set of equations $\{\a(\x,\y) \approx \b(\x,\y) : (\a,\b) \in \theta\}$. We claim that, for an arbitrary equation $\ep(\y,\z)$:
\begin{enumerate}
\item[(a)] $\ep(\y,\z) \in k^{-1} j^*(\theta) \iff \Si_\theta \mdl{\V} \ep$.
\item[(b)] $\ep(\y,\z) \in l^*i^{-1}(\theta) \iff$ there exists a finite set of equations $\Pi(\y)$ such that $\Si_\theta \mdl{\V} \Pi$ and $\Pi \mdl{\V} \ep$.
\end{enumerate}
(a) follows immediately from Lemma~\ref{lem:algcons}. For (b), recall that $\theta$ is the directed join of the compact congruences below it. This join is preserved by $i^{-1}$ and $l^*$ (since it is a left adjoint). Hence $\ep(\y,\z) \in l^*i^{-1}(\theta)$ if, and only if, there is a compact $\theta' \leq \theta$ such that $\ep(\y,\z) \in l^*i^{-1}(\theta')$. (b) now follows by Lemma~\ref{lem:algcons}.

The equivalence of (i) and (iii) follows from (a) and (b). (ii) $\Rightarrow$ (i) is an easy consequence of the finitarity of the equational consequence relation, so it remains to establish (i) $\Rightarrow$ (ii). Given finite sets $\x,\y$ and a finite set of equations $\Si(\x,\y)$, let $\Pi(\y) = \{ \de(\y) : \Si \mdl{\V} \de\}$. Clearly, $\Si \mdl{\V} \Pi$. So if $\Pi \mdl{\V} \ep$, then $\Si \mdl{\V} \ep$. Conversely, if $\Si \mdl{\V} \ep$, then, by (i), there exists a finite set of equations $\Pi'(\y)$ such that $\Si \mdl{\V} \Pi'$ and $\Pi' \mdl{\V} \ep$, so $\Pi' \subseteq \Pi$ and hence also $\Pi \mdl{\V} \ep$.
\end{proof}

Deductive interpolation enjoys close relationships with familiar algebraic properties. Recall that a class of algebras $\cls{K}$ has the {\em amalgamation property} if for every 5-tuple $(\m{A},\m{B},\m{C},i,j)$ where $\m{A},\m{B},\m{C}\in \cls{K}$ and $i,j$ are embeddings of $\m{A}$ into $\m{B},\m{C}$, respectively, there exists $\m{D} \in \cls{K}$ and embeddings $h,k$ of $\m{B}, \m{C}$, respectively, into $\m{D}$ such that the compositions $hi$ and $kj$ coincide. Recall also that $\cls{K}$ has the {\em congruence extension property} if any congruence of a subalgebra of $\m{A} \in \cls{K}$ extends to $\m{A}$. We state the following theorem, which appeared first in~\cite{Jon65}, credited there to unpublished work of H.J.~Keisler. 

\begin{theorem}[cf.~\cite{MMT14}*{Thm.~22}] \label{thm:amalgamation} \
\begin{enumerate}
\item[\rm (a)]
If $\V$ has the amalgamation property, then $\V$ admits deductive interpolation. 
\item[\rm (b)]
If $\V$ admits deductive interpolation and has the congruence extension property, then $\V$ has the amalgamation property.
\end{enumerate}
\end{theorem}

\begin{remark}
The congruence extension property for $\V$ is itself equivalent to the following property of equational consequence studied in~\cites{Ono86,Bac75,CP99,MMT14}: for any finite sets $\x,\y$ and finite set of equations $\Si(\x,\y) \cup \Pi(\y) \cup \{\ep(\y)\}$ satisfying $\Si,\Pi \mdl{\V} \ep$, there exists a finite set of equations $\De(\y)$ such that $\Si \mdl{\V} \De$ and $\De, \Pi \mdl{\V} \ep$ (see~\cite{MMT14}*{Thm.~20}).
\end{remark}

The combination of the amalgamation and congruence extension properties is equivalent to the following (stronger) interpolation property of equational consequence in $\V$.

\begin{definition}
$\V$ admits {\it Maehara interpolation} if for any finite sets $\x,\y,\z$ and finite sets of equations $\Ga(\x,\y), \Si(\y,\z), \De(\y,\z)$ satisfying $\Ga, \Si \mdl{\V} \De$, there exists a set of equations $\Pi(\y)$ such that $\Ga \mdl{\V} \Pi$ and $\Pi, \Si \mdl{\V} \De$.
\end{definition}

\begin{theorem}[cf.~\cite{MMT14}*{Thm.~29}] \label{thm:maehara}
The following are equivalent:
\begin{enumerate}
\item	$\V$ admits Maehara interpolation.
\item	$\V$ admits deductive interpolation and has the congruence extension property.
\item	$\V$ has the amalgamation property and the congruence extension property.
\end{enumerate}
\end{theorem}

The preceding theorem may be traced back to~\cite{Wro84a} and~\cite{Cze85}. Note also that the combination of the amalgamation property and the congruence extension property (and hence Maehara interpolation) is equivalent to the {\em transferable injections property} (see~\cite{MMT14}*{Sec.~5} for further details and references).


\section{Right uniform interpolation}\label{sec:right}

In this section, we give a suitable syntactic definition for right uniform deductive interpolation in the variety $\V$, and show it to be equivalent to a natural algebraic property of $\V$. 

\begin{definition} A variety 
$\V$ admits {\it right uniform deductive interpolation} if for any finite sets $\x,\y$ and finite set of equations $\Si(\x,\y)$, there exists a finite set of equations $\Pi(\y)$ such that 
\begin{equation}\label{eq:unint}
\Si \mdl{\V} \Pi \ \text{ and for any equation } 
\ep(\y,\z), \ \Si \mdl{\V} \ep \ \Longrightarrow \ \Pi \mdl{\V} \ep.
\end{equation}
A finite set of equations $\Pi(\y)$ that satisfies (\ref{eq:unint}) is called a {\it right uniform interpolant} of $\Si$ with respect to $\x$. 
\end{definition}

Since any two right uniform interpolants of $\Si$ with respect to $\x$ must be equational consequences of one another in $\V$, we may also speak of {\it the} right uniform interpolant, denoting by $\exists_{\x}(\Si)$ some finite set of equations $\Pi(\y)$ satisfying (\ref{eq:unint}).

A useful equivalent way of phrasing property (\ref{eq:unint}) is as follows:

\begin{equation}\label{eq:unintalt}
\text{for any equation } \ep(\y,\z), \ \Si \mdl{\V} \ep \iff \Pi \mdl{\V} \ep.
\end{equation}

\noindent
Note in particular that using this alternative characterization of deductive interpolation, it is easy to see that right uniform deductive interpolation implies deductive interpolation.

Uniform interpolation properties are closely related to properties of the \emph{finitely presented algebras} in the variety. Recall that a \emph{finite presentation} of an algebra $\m A$ is a surjective homomorphism $p \colon \F(\x) \onto \m A$ such that $\x$ is finite and $\ker p$ is a compact congruence on $\F(\x)$. An algebra $\m A$ is called \emph{finitely presented} if it has a finite presentation.

We now state the main theorem of this section.

\begin{theorem}\label{thm:rudipalgebraic}
The following are equivalent:
\begin{enumerate}
\item $\V$ admits right uniform deductive interpolation.
\item $\V$ admits deductive interpolation, and the compact lifting of any homomorphism between finitely presented algebras in $\V$ has a right adjoint.
\end{enumerate}
\end{theorem}

Observe that in a locally finite variety, any finitely presented algebra $\m{A}$ is finite and hence $\KCon \m{A} = \Con \m{A}$ for any finitely presented algebra $\m{A}$. So the second part of condition (ii) in Theorem~\ref{thm:rudipalgebraic} holds in any locally finite variety, and we obtain the following result.

\begin{corollary}\label{cor:rudiplocfin}
A locally finite variety admits right uniform deductive interpolation if, and only if, it admits deductive interpolation.
\end{corollary}

\begin{example}[Varieties generated by Heyting chains] \label{example:Heytingchains}
Any variety generated by a Heyting chain is locally finite, and admits deductive interpolation and hence also right uniform deduction interpolation if, and only if, the chain has one, two, three, or infinitely many elements~\cite{Mak77}. 
\end{example}

The proof of Theorem~\ref{thm:rudipalgebraic} will be split into two steps:\label{page:steps}

\begin{enumerate}
\item[(I)] Establish an equivalence between right uniform deductive interpolation and a property of congruences on free algebras (Proposition~\ref{prop:rudipintermediate}). 
\item[(II)] Show that the property of congruences on free algebras of step (I) extends to a property of all finitely presented algebras (Proposition~\ref{prop:inclusiontoall}).
\end{enumerate}

\noindent
Proving Propositions~\ref{prop:rudipintermediate} and \ref{prop:inclusiontoall}, from which Theorem~\ref{thm:rudipalgebraic} follows directly, will occupy us for the rest of this section. In the first of these propositions below, the equivalence of (i) and (iii) is required for Theorem~\ref{thm:rudipalgebraic}. The equivalent statement (ii) provides an interesting alternative characterization of right uniform deductive interpolation that involves only compact liftings of homomorphisms between (possibly infinitely generated) free algebras.

\begin{proposition}\label{prop:rudipintermediate}
The following are equivalent:
\begin{enumerate}
\item	$\V$ admits right uniform deductive interpolation. 
\item For any sets $\x,\y$, the compact lifting of the inclusion homomorphism $i \colon \F(\y) \into \F(\x,\y)$ has a right adjoint, $i^{-1}$. 
\item $\V$ admits deductive interpolation, and, for any finite sets $\x,\y$, the compact lifting of the inclusion homomorphism $i \colon \F(\y) \into \F(\x,\y)$ has a right adjoint, $i^{-1}$.
\end{enumerate}
\end{proposition}
\begin{proof}
(i) $\Rightarrow$ (ii). Suppose that $\V$ has right uniform deductive interpolation and consider any sets $\x,\y$ and $\theta \in \KCon \F(\x,\y)$. Pick a finite set of equations $\Si(\x,\y)$ that generates $\theta$. By right uniform deductive interpolation, we obtain a finite set of equations $\Pi(\y)$ such that  $\Si \mdl{\V} \ep$ if, and only if, $\Pi \mdl{\V} \ep$ for any equation $\ep(\y,\z)$. By Lemma~\ref{lem:algcons}, for any equation $\ep(\y)$, we have that $\Si \mdl{\V} \ep$ is equivalent to $i(\ep) \in \cg{\F(\x,\y)}(\Si) = \theta$, and $\Pi \mdl{\V} \ep$ is equivalent to $\ep \in \cg{\F(\y)}(\Pi)$. So $\cg{\F(\y)}(\Pi) = i^{-1}(\theta)$ and the congruence $i^{-1}(\theta)$ is compact, as required.

(ii) $\Rightarrow$ (iii). Suppose that (ii) holds. Clearly, to establish (iii), it suffices to prove that $\V$ admits deductive interpolation. Consider disjoint finite sets $\x$ and $\y$ and a finite set of equations $\Si(\x,\y)$. Let $\w$ be an infinite set disjoint to $\x$ and $\y$ and let $\theta := \cg{\F(\x,\y,\w)}(\Si)$. By assumption, the compact lifting of the inclusion homomorphism $i \colon \F(\y,\w) \into \F(\x,\y,\w)$ has a right adjoint $i^{-1}$. Let $\De(\y,\uvar)$ be a finite set of equations that generates $i^{-1}(\theta)$ for some finite $\uvar \subseteq \w$. Since $\w$ is infinite and $\uvar$ is finite, we may identify $\z$ in the statement of deduction interpolation using condition~\eqref{eq:unintalt} with the infinite set $\w \setminus \uvar$. Let $\si \colon \Tm(\x,\y,\w) \to \Tm(\x,\y,\z)$ be any substitution acting as the identity on $\Tmc(\x,\y,\z)$ and satisfying $\si(\Tmc(\uvar)) \subseteq \Tmc(\y)$. We define $\Pi(\y) := \si(\De)$. 

Observe now that for any equation $\ep(\y,\w)$,
\[
i(\ep) \in \cg{\F(\x,\y,\w)}(\Si) = \theta \iff \ep \in i^{-1}(\theta) = \cg{\F(\y,\w)}(\De)
\]
and hence, using Lemma~\ref{lem:algcons},
\[
\Si \mdl{\V} \ep \iff \De \mdl{\V} \ep.
\]
In particular, $\Si \mdl{\V} \De$. Consider an equation $\ep(\y,\z)$ such that $\Si \mdl{\V} \ep$. Then $\De \mdl{\V} \ep$, and so $\si(\De) \mdl{\V} \si(\ep)$, that is, $\Pi \mdl{\V} \ep$. Suppose, conversely, that $\Pi \mdl{\V} \ep$. Since $\Si \mdl{\V} \De$, it follows that $\si(\Si) \mdl{\V} \si(\De)$, that is, $\Si \mdl{\V} \Pi$ and, combining these consequences, $\Si \mdl{\V} \ep$. So $\V$ admits deductive interpolation.

(iii) $\Rightarrow$ (i). Let $\x,\y$ be finite sets and $\Si(\x,\y)$ a finite set of equations. Define $\theta = \cg{\F(\x,\y)}(\Si)$. Using (iii), the compact lifting of the inclusion homomorphism $i \colon \F(\y) \into \F(\x,\y)$ has a right adjoint $i^{-1}$. Let $\Pi(\y)$ be a finite set of equations that generates $i^{-1}(\theta)$. Now consider any equation $\ep(\y,\z)$. If $\Si \mdl{\V} \ep$, then, using deductive interpolation, pick a finite set of equations $\De(\y)$ such that $\Si \mdl{\V} \De$ and $\De \mdl{\V} \ep$. Since only variables from $\y$ occur in $\De$ and $i^{-1}(\cg{\F(\x,\y)}(\Si)) = \cg{\F(\x)}(\Pi)$, we have $\Pi \mdl{\V} \De$. Hence $\Pi \mdl{\V} \ep$.
\end{proof}

\begin{example}[Groups]
The variety of groups has the amalgamation property and therefore admits deductive interpolation, but it does not admit right uniform deductive interpolation. It is not the case that for any finite sets $\x,\y$, the compact lifting of the inclusion homomorphism $i \colon \F(\y) \into \F(\x,\y)$ has a right adjoint. That is, there exists $\theta \in \KCon \F(\x,\y)$ such that $\theta \cap \F(\y)^2$ is not finitely generated. Indeed, let $p \colon \F(\y) \onto \m{G}$ be a recursive presentation of a finitely generated group $\m{G}$ which is not finitely presentable. (For example, take $\y = \{a,b\}$, and $\m{G} = \F(\y)/\gen{R}$, where $R$ is the relation $\{(a^{-n}ba^nb,ba^{-n}ba^n) : n \in \mathbb{N}\}$.) By Higman's embedding theorem, there exists a finitely presented group $\m{H}$ and an embedding $j \colon \m{G} \into \m{H}$. Choose a finite generating set $A \subseteq H$ which contains $jp(y)$ for every $y \in \y$. Let $\x$ be a finite set of variables disjoint from $\y$ such that there is a bijection $\x \cup \y \to A$, and let $q \colon \F(\x,\y) \onto \m{H}$ be the unique homomorphism extending this bijection. By Lemma~\ref{lem:fingen} below, $\theta := \ker q$ is compact, since $H$ is finitely presented. But $\theta \cap \F(\y)^2 = \ker p$, which is not compact by assumption.\footnote{This example is based on an argument communicated to us by M. Sapir.} 
\end{example}

\begin{example}[Lattice-ordered abelian groups and MV-algebras]\label{exa:ablgpsandmv}
We claim first that the variety $\mathcal{LA}$ of lattice-ordered abelian groups, generated as a quasivariety by $\langle \mathbb{R},\land,\lor,+,-,0 \rangle$ (cf. \cite{AF88}*{Lemma~6.2}), admits right uniform deductive interpolation. It suffices to show that a finite set of equations $\Si(x,\y)$ has a right uniform interpolant with respect to $x$. Moreover, writing $\a \le \b$ for $\a \land \b \eq \a$, we may assume (with a little work, omitted here) that $\Si$ consists of inequations $0 \le \a_i + nx$ ($i \in I$), $0 \le \b_j - nx$ ($j \in J$), and $0 \le \ga_k$ ($k \in K$) for some $n \ge 1$, finite sets $I,J,K$, and terms $\a_i,\b_j,\ga_k$ that do not contain $x$. A right uniform interpolant $\Pi(\y)$ then consists of the inequations $0 \le \a_i + \b_j$ ($i \in I, j \in J$) and $0 \le \ga_k$ ($k \in K$). Since $\mathcal{LA}$ also has the congruence extension property, this variety has the amalgamation property. 

Let us consider now the variety $\mathcal{MV}$ of MV-algebras for a signature $\lang_\mathcal{MV}$, generated as a quasivariety by $\m{[0,1]} = \langle [0,1],\oplus,\lnot,0 \rangle$, where $a \oplus b = \min(1,a+b)$ and $\lnot a = 1-a$. We recall that further operations may be defined as $1 := \lnot 0$, $a \odot b := \lnot (\lnot a \oplus \lnot b)$, $a \lor b := \lnot (\lnot a \oplus b) \oplus b$, and $a \land b := \lnot (\lnot a \lor \lnot b)$. Consider the signature $\lang_\mathcal{A}$ of lattice-ordered abelian groups with an additional constant $1$ and the $\lang_\mathcal{A}$-algebra $\m{R} = \langle \mathbb{R},\land,\lor,+,-,0,1 \rangle$. It follows from McNaughton's representation theorem (or see~\cite{MMT14}*{Sec.~6} for a direct proof) that (i) the interpretation of any $\lang_\mathcal{MV}$-term $\b$ in $\m{[0,1]}$ is equivalent to the interpretation of some $\lang_\mathcal{A}$-term $(\a \land 0) \lor 1$ in $\m{R}$, and, conversely, (ii) the interpretation of any $\lang_\mathcal{A}$-term $(\a \land 0) \lor 1$ in $\m{R}$ is equivalent to the interpretation of some $\lang_\mathcal{MV}$-term $\b$ in $\m{[0,1]}$. Right uniform deductive interpolation for $\mathcal{MV}$ may then be established as in the case of lattice-ordered abelian groups described above. Again, since $\mathcal{MV}$ also has the congruence extension property, this variety has the amalgamation property. 

Further details of these arguments, without explicit mention of uniform interpolation, may be found in~\cite{MMT14}*{Sec.~6}. Let us also remark that the representation of finitely presented MV-algebras via rational polyhedra can be used to provide a geometric argument for uniform interpolation (see~\cite{Mun11}).
\end{example}

\begin{proposition}\label{prop:inclusiontoall}
The following are equivalent:
\begin{enumerate}
\item For any finite sets $\x,\y$, the compact lifting of the inclusion homomorphism $i \colon \F(\x) \into \F(\x,\y)$ has a right adjoint.
\item For any finitely presented algebras $\m A$, $\m B$ in $\V$, the compact lifting of any homomorphism $f \colon \m A \to \m B$ has a right adjoint.
\end{enumerate}
\end{proposition}

The key point in proving (i) $\Rightarrow$ (ii) in Proposition~\ref{prop:inclusiontoall} will be to show, in Lemma~\ref{lem:preschoice}, that it is always possible to choose suitable finite presentations of the finitely presented algebras $\m{A}$ and $\m{B}$. Before we can do this, we need the following preliminary lemma, which shows that, in a finitely presented algebra, we can arbitrarily choose a finite set of generators, and always obtain a finite presentation. The proof generalizes an argument from group theory due to B.H.~Neumann, cf., e.g.,~\cite{Rob95}*{Thm.~2.2.3}. These results are also closely related to \cite{Coh81}*{Theorem~III.8.4}. We give direct algebraic proofs.

\begin{lemma}\label{lem:fingen}
Let $\m{A}$ be a finitely presented algebra in $\V$. For any finite set $\x$ and surjective homomorphism $f \colon \F(\x) \onto \m{A}$, the congruence $\ker f$ is compact.
\end{lemma}

\begin{proof}
Pick a finite presentation $g \colon \F(\y) \onto \m{A}$ of $\m{A}$, and a finite set $\Pi(\y)$ of generators for $\ker g$. Since $g$ is surjective, we may choose a substitution $t \colon \F(\x) \to \F(\y)$ such that $g \circ t = f$, and, as $f$ is surjective, also a substitution $s \colon \F(\y) \to \F(\x)$ such that $f \circ s = g$.

{\bf Claim.} The finite set $\{(s(\a),s(\b)) : (\a,\b) \in \Pi\} \cup \{(x,st(x)) : x \text{ in } \x\}$ generates $\ker f$.

{\it Proof of Claim.} Denote by $\psi$ the congruence generated by the set in the claim.
Clearly, $\psi \subseteq \ker f$.
For the other inclusion, let $\m{B} := \F(\x)/\psi$ and denote by $p \colon \F(\x) \onto \m{B}$ the natural quotient map.
Let $h \colon \F(\y) \to \m{B}$ be the homomorphism $p \circ s$.
Note that, for each $(\a,\b) \in \Pi$, 
\[ 
h(\a) = p(s(\a)) = p(s(\b)) = h(\b).
\]
Hence $\ker g \subseteq \ker h$, since $\Pi$ generates $\ker g$. The general homomorphism theorem then implies that there exists a unique homomorphism $k \colon\m{A}\to \m{B}$ such that $k \circ g = h$.
For each $x$ in $\x$, 
\[
p(x) = p(s(t(x)) = h(t(x)) = k(g(t(x))) = k(f(x)).
\]
Since both $p$ and $k \circ f$ are homomorphisms from $\F(\x)$ to $\m{B}$, it now follows that $p = k \circ f$. Hence, $\ker f \subseteq \ker p = \psi$.
\end{proof}

We may now easily conclude the following results.

\begin{corollary}\label{cor:finpres}
Let $\m{A}$ be a finitely presented algebra in $\V$ and let $\theta$ be a congruence on $\m{A}$ such that $\m{A}/{\theta}$ is finitely presented. Then $\theta$ is compact.
\end{corollary}
\begin{proof}
Let $f \colon \F(\x) \onto \m{A}$ be a finite presentation of $\m{A}$. Let $p$ be the quotient map $\m{A} \onto \m{A}/{\theta}$ and let $g := p \circ f$. By Lemma~\ref{lem:fingen}, $f^{-1}(\theta) = \ker g$ is a compact congruence. Since $f$ is surjective, we have $f^{*}(f^{-1}(\theta)) = \theta$. Hence, since $f^*$ preserves compact congruences, $\theta$ is compact.
\end{proof}

\begin{lemma}\label{lem:preschoice}
Let $f \colon\m{A}\to \m{B}$ be a homomorphism between finitely presented algebras in $\V$. Then there exist finite presentations ${p_A \colon \F(\x) \onto \m{A}}$ and ${p_B \colon \F(\x,\y) \onto \m{B}}$ such that $p_B \circ i = f \circ p_A$, where $i \colon \F(\x) \into \F(\x,\y)$ is the inclusion homomorphism.
\end{lemma}
\begin{proof}
Let $p_A \colon \F(\x) \onto \m{A}$ and $q_B \colon \F(\y) \onto \m{B}$ be finite presentations of $\m{A}$ and $\m{B}$.
Since $q_B$ is surjective, for each $x \in \x$, we may pick $s(x) \in \F(\y)$ such that $f(p_A(x)) = q_B(s(x))$. Define a substitution $s \colon \F(\x,\y) \to \F(\y)$ by sending each $y$ in $\y$ to itself, and each $x$ in $\x$ to $s(x)$. Let $p_B := q_B \circ s$. Note that $p_B$ is onto and that $p_B \circ i = f \circ p_A$ holds by construction. Moreover, $\F(\x,\y)/{\ker p_B}$ is isomorphic to the finitely presented algebra $\m B$, so, by Corollary~\ref{cor:finpres}, $\ker p_B$ is compact.
\end{proof}

We are now in a position to prove Proposition~\ref{prop:inclusiontoall}, which also concludes the proof of Theorem~\ref{thm:rudipalgebraic}.
\begin{proof}[Proof of Proposition~\ref{prop:inclusiontoall}]
(ii) $\Rightarrow$ (i) is trivial. 
For (i) $\Rightarrow$ (ii) let $f \colon\m{A}\to \m{B}$ be a homomorphism between finitely presented algebras. We prove that $f^{-1}(\theta)$ is compact for any compact congruence $\theta$ on $\m{B}$. By Lemma~\ref{lem:preschoice}, there exist finite presentations $p_A \colon \F(\x) \onto A$ and $p_B \colon \F(\x, \y) \onto \m{B}$ such that $p_B \circ i = f \circ p_A$. Now, if $\theta$ is a compact congruence on $\m{B}$, then $p_B^{-1}(\theta)$ is a compact congruence on $\F(\x,\y)$ by Corollary~\ref{cor:finpres}, since $\m{B}/\theta$ is finitely presented. By assumption, the congruence $\psi := i^{-1}(p_B^{-1}(\theta))$ is compact. Note that, since $p_B \circ i = f \circ p_A$, we have $\psi = p_A^{-1}(f^{-1}(\theta))$. Hence the composite $\F(\x) \onto\m{A}\onto \m{A}/{f^{-1}(\theta)}$ is a finite presentation of the algebra $\m{A}/{f^{-1}(\theta)}$. By Corollary~\ref{cor:finpres} again, $f^{-1}(\theta)$ is compact. 
\end{proof}

Using Theorem~\ref{thm:rudipalgebraic}, we also obtain an algebraic characterization of a right uniform version of Maehara interpolation.

\begin{definition}
$\V$ admits {\it right uniform Maehara interpolation} if for any finite sets $\x,\y$ and a finite set of equations $\Ga(\x,\y)$, there exists a finite set of equations $\Pi(\y)$ such that for any finite sets of equations $\Si(\y,\z),\De(\y,\z)$,
\[
\Ga, \Si \mdl{\V} \De \iff \Pi, \Si \mdl{\V} \De.
\]
\end{definition}

\begin{theorem}\label{thm:umaehararight}
The following are equivalent:
\begin{enumerate}
\item	$\V$ admits right uniform Maehara interpolation. 
\item	$\V$ admits right uniform deductive interpolation and has the congruence extension property. 
\item	$\V$ has the amalgamation property and the congruence extension property, and the compact lifting of any homomorphism between finitely presented algebras in $\V$ has a right adjoint.
\end{enumerate}
\end{theorem}
\begin{proof}
(i) $\Rightarrow$ (ii). If $\V$ admits right uniform Maehara interpolation, then it clearly admits both right uniform deductive interpolation and Maehara interpolation. By Theorem~\ref{thm:maehara}, $\V$ also has the congruence extension property. 

(ii) $\Rightarrow$ (iii). Suppose that $\V$ admits right uniform deductive interpolation and has the congruence extension property. It follows that $\V$ admits deductive interpolation, and hence, by Theorem~\ref{thm:amalgamation}, has the amalgamation property. That the compact lifting of any homomorphism between finitely presented algebras in $\V$ has a right adjoint is a consequence of Theorem~\ref{thm:rudipalgebraic}.

(iii) $\Rightarrow$ (i). Suppose first that $\V$ has the amalgamation property and the congruence extension property. Then, by Theorem~\ref{thm:maehara}, $\V$ admits Maehara interpolation. Suppose also that the compact lifting of any homomorphism between finitely presented algebras in $\V$ has a right adjoint. Then, by Theorem~\ref{thm:rudipalgebraic}, $\V$ admits right uniform deductive interpolation. Now consider finite sets $\x,\y$ and a finite set of equations $\Ga(\x,\y)$. By right uniform deductive interpolation, there exists a finite set of equations $\Pi(\y)$ such that for any finite sets of equations $\De(\y,\z)$,
\[
\Ga \mdl{\V} \De \iff \Pi \mdl{\V} \De.
\]
Now consider any finite sets of equations $\Si(\y,\z),\De(\y,\z)$. If $\Ga, \Si \mdl{\V} \De$, then by Maehara interpolation, there exists a finite set of equations $\Pi'(\y)$ such that $\Ga \mdl{\V} \Pi'$ and $\Pi', \De \mdl{\V} \Si$. But then also $\Pi \mdl{\V} \Pi'$ and, as required, $\Pi, \De \mdl{\V} \Si$. Conversely, suppose $\Pi, \De \mdl{\V} \Si$. Then $\Ga \mdl{\V} \Pi$, by the preceding equivalence, and hence $\Ga, \Si \mdl{\V} \De$. 
\end{proof}


\section{Left uniform interpolation}\label{sec:left}

In this section, we introduce a uniform deductive interpolation property on the left, observing that left and right uniform interpolation do not behave entirely symmetrically. In Proposition~\ref{prop:ludipintermediate}, we provide an analogue of the characterization in Proposition~\ref{prop:rudipintermediate}, but then see in Example~\ref{exa:leftdoesnotgeneralize} that the direct analogue of Theorem~\ref{thm:rudipalgebraic} does not hold for left uniform interpolation. We therefore provide an alternative characterization of this property in Theorem~\ref{thm:ludipalgebraic}.

\begin{definition}
$\V$ admits {\it left uniform deductive interpolation} if for any finite sets $\y,\z$ and finite set of equations $\De(\y,\z)$, there exists a finite set of equations $\Pi(\y)$ such that
\begin{equation}\label{eq:unintleft}
\Pi \mdl{\V} \De \ \text{ and for any set of equations } \Si(\x,\y), \ \Si \mdl{\V} \De \ \Rightarrow \ \Si \mdl{\V} \Pi.
\end{equation}
A finite set of equations $\Pi(\y)$ satisfying (\ref{eq:unintleft}) is called a {\it left uniform interpolant} of $\De$ with respect to the variables $\z$.
\end{definition}

Again, two left uniform interpolants of the same set $\De$ are equational consequences of one another, so we may also speak of {\it the} left uniform interpolant, denoting by $\forall_{\z}(\De)$ some finite set of equations $\Pi(\y)$ satisfying (\ref{eq:unintleft}). We also again obtain a useful equivalent way of phrasing property (\ref{eq:unintleft}):
\begin{equation}\label{eq:unintleftalt}
\text{for any set of equations } \Si(\x,\y), \ \Si \mdl{\V} \De \iff \Si \mdl{\V} \Pi.
\end{equation}

\begin{remark}\label{rem:trivialfailure}
Left uniform deduction interpolation may fail for a variety for a rather trivial reason. It may be the case that for some finite set of equations $\De(\y,\z)$, there is no set of equations $\Pi(\y)$ satisfying $\Pi \mdl{\V} \De$. For example, in the variety of lattice-ordered abelian groups $\mathcal{LA}$, clearly there exists no set of equations $\Pi(\y)$ such that $\Pi \mdl{\mathcal{LA}} y \eq z$. 
\end{remark}

The proof of the following proposition is entirely analogous to that of Proposition~\ref{prop:rudipintermediate}.

\begin{proposition}\label{prop:ludipintermediate}
The following are equivalent:
\begin{enumerate}
\item $\V$ admits left uniform deductive interpolation. 
\item For any sets $\x,\y$, the compact lifting of the inclusion homomorphism $i \colon \F(\y) \hookrightarrow \F(\x, \y)$ has a left adjoint $\forall_i$.
\item $\V$ admits deductive interpolation, and, for any finite sets $\x,\y$, the compact lifting of the inclusion homomorphism $i \colon \F(\y) \into \F(\x,\y)$ has a left adjoint $\forall_i$.
\end{enumerate}
\end{proposition}

Recalling that an order-preserving map between complete lattices has a left adjoint if, and only if, it preserves all meets, we obtain the following result.

\begin{corollary}\label{cor:leftlocfin}
If $\V$ is locally finite, then the following are equivalent:
\begin{enumerate}
\item $\V$ admits left uniform deductive interpolation.
\item For any sets $\x,\y$, the compact lifting of the inclusion homomorphism $i \colon \F(\y) \into \F(\x,\y)$ preserves intersections.
\item $\V$ admits deductive interpolation, and, for any finite sets $\x,\y$, the compact lifting of the inclusion homomorphism $i \colon \F(\y) \into \F(\x,\y)$ preserves intersections.
\end{enumerate}
\end{corollary}

\begin{example}[Implicative semilattices] 
The variety $\cls{ISL}$ of implicative semilattices is locally finite and admits deductive interpolation. Hence it admits right uniform deductive interpolation. However, it is not the case that for any finite sets $\x,\y$, the lifting of the inclusion homomorphism $i \colon \F(\y) \into \F(\x,\y)$ preserves intersections. For example, let $\x = \{x\}$, $\y = \{y_1,y_2\}$ and let $\theta_i$ be the congruence on $\F(\y)$ generated by $(\top,y_i)$ for $i = 1,2$. Then $(\top, ((y_1 \to x) \land (y_2 \to x)) \to x)$ is in $i^*(\theta_1) \cap i^*(\theta_2)$, but not in $i^*(\theta_1 \cap \theta_2)$ (otherwise, the join operation would be definable in $\cls{ISL}$, which it is not). So $\cls{ISL}$ does not admit left uniform deductive interpolation.
\end{example}

As mentioned already, the analogue of Theorem~\ref{thm:rudipalgebraic} does not hold for the case of left uniform deductive interpolation. How might we fix this? Recall that, in Section~\ref{sec:right}, after proving the `right' version of Proposition~\ref{prop:ludipintermediate}, which we called step (I) on page~\pageref{page:steps}, we were able to extend the relevant property from inclusion maps to arbitrary homomorphisms between finitely presented algebras. Example~\ref{exa:leftdoesnotgeneralize} shows that the same strategy cannot work for the left case. It turns out that the `missing ingredient' for proving a characterization as in Theorem~\ref{thm:rudipalgebraic} for the left case is a characterization of the maps of the form $p^*$, for $p$ a \emph{surjective} homomorphism between finitely presented algebras, that have a left adjoint.

For this missing ingredient, recall that a join-semilattice $\langle L,\lor,\bot \rangle$ is called \emph{dually Brouwerian} if the operation $\lor$ is left residuated, i.e., for any $a,b \in L$, there exists an element $a - b \in L$ such that for any $c \in L$, $a - b \leq c$ if, and only if, $a \leq b \lor c$. The following lemma is crucial.

\begin{lemma}\label{lem:brouwerian}
For any algebra $\m{A}$, the following are equivalent:
\begin{enumerate}
\item For all $\theta \in \KCon \m{A}$, the compact lifting of $p_\theta \colon \m{A} \onto \m{A}/{\theta}$ has a left adjoint, $\forall_p$.
\item The join-semilattice $\KCon \m{A}$ is dually Brouwerian.
\end{enumerate}
\end{lemma}

\begin{proof} 
Note that a semilattice $\m{L}$ is dually Brouwerian, if, and only if, for every $a \in \m{L}$, the function $j_a \colon \m{L} \onto {\uparrow}_{\m{L}}a$ defined by $j_a(b) := a \lor b$, has a left adjoint. We prove that the latter condition is equivalent to (i) when $\m{L} = \KCon \m{A}$. By Proposition~\ref{prop:surjective}, for any $\theta \in \KCon\m{A}$, the following diagram is well-defined and commutes:

\begin{center}
\begin{tikzpicture}[baseline=(current bounding box.center)]
\matrix (m) [matrix of math nodes, row sep=1.5em, column sep=1.5em]{
& \KCon \m{A}	& \\
&	& \\
\KCon \m{A}/\theta & & \text{[}\theta,\nabla_{\m A}\text{]}_{\KCon\m A} \\
};
\path
(m-1-2) edge[->] node[left,xshift=-2mm] {$p_\theta^*$} (m-3-1)
(m-1-2) edge[->] node[right, xshift=2mm] {$j_\theta$} (m-3-3)
(m-3-1) edge[->] node[above] {$p_\theta^{-1}$} node[below] {$\cong$} (m-3-3);
\end{tikzpicture}
\end{center}

\noindent
Since $p_\theta^{-1}$ is an isomorphism, $p_\theta^*$ has a left adjoint if, and only if, $j_\theta$ has a left adjoint. 
\end{proof}

\begin{proposition}\label{prop:finpresbrouwerian}
The following are equivalent:
\begin{enumerate}
\item The compact lifting of any finite presentation $p \colon \F(\x) \onto \m A$ has a left adjoint. 
\item For any finite set $\x$, the join-semilattice $\KCon \F(\x)$ is dually Brouwerian.
\item For any finitely presented $\m{A} \in \V$, the join-semilattice $\KCon \m{A}$ is dually Brouwerian. 
\item For any finite set $\x$ and finite sets of equations $\Si(\x),\De(\x)$, there exists a finite set of equations $\Pi(\x)$ such that for any finite set of equations $\Ga(\x)$,
\[
\Ga, \Si \mdl{\V} \De \iff \Ga \mdl{\V} \Pi.
\]
\end{enumerate}
If $\V$ is locally finite, then (i)-(iv) are equivalent to
\begin{enumerate}
\item[(v)] $\V$ is congruence distributive. 
\end{enumerate}
\end{proposition}

\begin{proof}
The equivalence of (i) and (ii) follows immediately from Lemma~\ref{lem:brouwerian}. Also (iii) clearly implies (ii). For the converse direction, let $p \colon \F(\x) \to \m A$ be a finite presentation of $\m A$. By the Correspondence Theorem~\cite{BS81}*{Thm.~6.20}, $\Con \m A \cong [\ker p, \nabla]_{\Con \F(\x)}$. Then, using Lemma~\ref{lem:compininterval}, $\K([\ker p,\nabla]_{\Con \F(\x)}) = [\ker p, \nabla]_{\KCon \F(\x)}$, so $\KCon \m A \cong [\ker p, \nabla]_{\KCon \F(\x)}$. From the assumption that $\KCon \F(\x)$ is dually Brouwerian, it follows that the join-semilattice $[\ker p, \nabla]_{\KCon \F(\x)}$ is also dually Brouwerian. 

The equivalence of (ii) and (iv) follows directly from the correspondence between equational consequence and congruences on free algebras established in Lemma~\ref{lem:algcons}. Finally, observe that if $\V$ is locally finite, then $\F(\x)$ is finite for any finite set $\x$ and the finite lattice $\KCon \F(\x) = \Con \F(\x)$ is dually Brouwerian if, and only if, it is distributive, so (ii) is equivalent to (v).
\end{proof}

\begin{remark}
The property that $\KCon \m{A}$ is dually Brouwerian for all $\m{A} \in \V$ is equivalent to the property that $\V$ admits {\em equationally definable principal congruences} (see~\cites{KP80,BKP84} for details and further characterizations). This is strictly stronger than the property that $\KCon \m{A}$ is dually Brouwerian for all finitely presented $\m{A} \in \V$. In particular, it is known that the only non-trivial variety of lattices admitting equationally definable principal congruences is the variety of distributive lattices~\cite{McK78}. On the other hand, any variety of lattices generated by a finite lattice (in particular, any finite non-distributive lattice) is locally finite and congruence distributive. 

Let us mention also that if $\KCon \m{A}$ is dually Brouwerian for some algebra $\m{A}$, then $\Con \m{A}$ is distributive. Note first that $\Con \m{A}$ is always isomorphic to the ideal lattice $\m{L}$ of the join-semilattice $\KCon\m{A}$. Moreover, if $\KCon \m{A}$ is dually Brouwerian, then $\m{L} \cong \Con \m{A}$. But the variety of dually Brouwerian join-semilattices is congruence distributive, so also $\Con \m{A}$ is distributive. These results about dually Brouwerian join-semilattices, stated here without proof, may be found in~\cites{Nem69,NemWha71}.
\end{remark}

\begin{example}[MV-algebras]\label{exa:leftdoesnotgeneralize}
Recall from Example~\ref{exa:ablgpsandmv} that the variety $\mathcal{MV}$ of MV-algebras admits right uniform interpolation. A similar argument demonstrates that it also admits left uniform interpolation; in particular, unlike the case of lattice-ordered abelian groups mentioned in Example~\ref{rem:trivialfailure}, there is always a variable-free set of equations satisfying $\Pi \mdl{\mathcal{MV}} \De$ for any set of equations $\De$, namely $\Pi = \{0 \eq \lnot 0\}$. Hence for any sets $\x,\y$, the compact lifting of the inclusion homomorphism $i \colon \F(\y) \hookrightarrow \F(\x, \y)$ has a left adjoint $\forall_i$. However, it is not the case that the compact lifting of {\em any} finite presentation $p \colon \F(\x) \onto \m A$ has a left adjoint. 

We show that the equivalent syntactic condition (iv) of Proposition~\ref{prop:finpresbrouwerian} fails. Recall first the local deduction theorem for MV-algebras stating that for any set of equations $\Ga$ and terms $\a,\b$,
\[
\Ga \cup \{\a \eq 1\} \mdl{\mathcal{MV}} \b \eq 1 \iff \Ga \mdl{\mathcal{MV}} \a^n \le \b \text{ for some }n \in \mathbb{N}.
\] 
Now let $\Si = \{x \eq 1\}$ and $\De = \{y \eq 1\}$. Suppose that there exists $\Pi(x,y)$ such that for any finite set of equations $\Ga(x,y)$,
\[
\Ga, \Si \mdl{\mathcal{MV}} \De \iff \Ga \mdl{\mathcal{MV}} \Pi.
\]
Then $\Pi, \Si \mdl{\mathcal{MV}} \De$, so, by the local deduction theorem, $\Pi \mdl{\mathcal{MV}} x^n \le y$ for some $n \in \mathbb{N}$. However, also $\{x^{n+1} \le y\} \cup \Si \mdl{\mathcal{MV}} \De$, so $\{x^{n+1} \le y\} \mdl{\mathcal{MV}} \Pi$. It follows that $\{x^{n+1} \le y\} \mdl{\mathcal{MV}} x^n \le y$, a contradiction. 
\end{example}

We now prove our main theorem for left uniform deductive interpolation. The proof is reasonably straightforward at this point, thanks to the groundwork laid in Sections~\ref{sec:prelim}~and~\ref{sec:right}. 

\begin{theorem}\label{thm:ludipalgebraic}
The following are equivalent:
\begin{enumerate}
\item $\V$ admits left uniform deductive interpolation, and $\KCon \F(\x)$ is a dually Brouwerian join-semilattice for any finite set $\x$. 
\item $\V$ admits deductive interpolation, and the compact lifting of any homomorphism between finitely presented algebras in $\V$ has a left adjoint.
\end{enumerate}
\end{theorem}

\begin{proof}
(i) $\Rightarrow$ (ii). We have already observed that deductive interpolation follows from left uniform deductive interpolation. Consider a homomorphism $f \colon \m{A} \to \m{B}$ between finitely presented algebras. Using Lemma~\ref{lem:preschoice}, we may pick finite presentations $p_A \colon \F(\x) \to \m{A}$ and $p_B \colon \F(\x,\y) \to \m{B}$ such that $p_B \circ i = f \circ p_A$. Then the following lifted diagram commutes:
\begin{equation}
\begin{tikzpicture}[baseline=(current bounding box.center)]
\matrix (m) [matrix of math nodes, row sep=1.5em, column sep=1.5em]{
\KCon \F(\x)	& & \KCon \F(\x,\y) \\
&	& \\
\KCon \m{A}	& & \KCon \m{B}	\\
};

\path
(m-1-1) edge[->] node[above] {$i^{*}$} (m-1-3)
(m-1-1) edge[->] node[left] {$(p_A)^*$} (m-3-1)
(m-3-1) edge[->] node[below] {$f^*$} (m-3-3)
(m-1-3) edge[->] node[right] {$(p_B)^*$} (m-3-3);
\end{tikzpicture}
\end{equation}
By Proposition~\ref{prop:surjective}(c), $(p_A)^{-1}$ preserves compact congruences. 
Moreover, $(p_A)^* \circ (p_A)^{-1} = \mathrm{id}_{\KCon \m{A}}$. So we obtain 
\[
f^* = f^* \circ (p_A)^* \circ (p_A)^{-1} = (p_B)^* \circ i^* \circ (p_A)^{-1}.
\]
Recall from Section~\ref{sec:prelim} that $(p_A)^{-1}$ always has a left adjoint, $(p_A)^*$. Moreover, $i^*$ has a left adjoint by (i) and Proposition~\ref{prop:ludipintermediate}, and $(p_B)^*$ has a left adjoint by (i) and Proposition~\ref{prop:finpresbrouwerian}; we denote these left adjoints by $\forall_i$ and $\forall_{p_B}$, respectively. Hence, since adjunctions compose, $f^* = (p_B)^* \circ i^* \circ (p_A)^{-1}$ has a left adjoint, namely $\forall_f := (p_A)^* \circ \forall_i \circ \forall_{p_B}$. 

(ii) $\Rightarrow$ (i). Follows using Proposition~\ref{prop:ludipintermediate} and Proposition~\ref{prop:finpresbrouwerian}.
\end{proof}

As in the case for right uniform deductive interpolation, we also obtain an algebraic characterization of a left uniform version of Maehara interpolation.

\begin{definition}
$\V$ admits {\it left uniform Maehara interpolation} if for any finite sets $\y,\z$ and finite sets of equations $\Si(\y,\z),\De(\y,\z)$, there exists a finite set of equations $\Pi(\y)$ such that for any finite set of equations $\Ga(\x,\y)$,
\[
\Ga, \Si \mdl{\V} \De \iff \Ga \mdl{\V} \Pi.
\]
\end{definition}

\begin{theorem}\label{thm:umaeharaleft}
The following are equivalent:
\begin{enumerate}
\item $\V$ admits left uniform Maehara interpolation. 
\item	$\V$ has the amalgamation property and the congruence extension property, and the compact lifting of any homomorphism between finitely presented algebras in $\V$ has a left adjoint.
\end{enumerate}
\end{theorem}
\begin{proof}
(i) $\Rightarrow$ (ii). Suppose that $\V$ admits left uniform Maehara interpolation. Then $\V$ clearly admits Maehara interpolation and left uniform deductive interpolation. Hence, by Theorem~\ref{thm:maehara}, $\V$ has the amalgamation property and the congruence extension property. Moreover, using the equivalent syntactic property given in Proposition~\ref{prop:finpresbrouwerian}, $\KCon \F(\x)$ is dually Brouwerian for any finite set $\x$. So by the previous theorem, the compact lifting of any homomorphism between finitely presented algebras in $\V$ has a left adjoint.

(ii) $\Rightarrow$ (i). Suppose that $\V$ has the amalgamation property and the congruence extension property (equivalently Maehara interpolation) and that the compact lifting of any homomorphism between finitely presented algebras in $\V$ has a left adjoint. Then $\V$ admits  deductive interpolation, so, by the previous theorem, $\V$ admits left uniform deductive interpolation and $\KCon \F(\x)$ is a dually Brouwerian join-semilattice for any finite set $\x$. Now consider any finite sets $\y,\z$ and finite sets of equations $\Si(\y,\z),\De(\y,\z)$. Using the fact that $\KCon \F(\y,\z)$ is dually Brouwerian, we obtain a finite set of equations $\Pi'(\y,\z)$ such that for any finite set of equations $\Ga'(\y,\z)$,
\[
\Ga', \Si \mdl{\V} \De \iff \Ga' \mdl{\V} \Pi'.
\]
Now, using left uniform deductive interpolation, we obtain a finite set of equations $\Pi(\y)$ such that for any finite set of equations $\Ga(\x,\y)$,
\[
\Ga \mdl{\V} \Pi' \iff \Ga \mdl{\V} \Pi.
\]
Consider any finite set of equations $\Ga(\x,\y)$. Observe first that, if $\Ga \mdl{\V} \Pi$, then $\Ga \mdl{\V} \Pi'$. But also $\Pi', \Si \mdl{\V} \De$, so, as required, $\Ga, \Si \mdl{\V} \De$. Suppose, conversely, that $\Ga, \Si \mdl{\V} \De$. Using Maehara interpolation, we obtain a finite set of equations $\Pi''(\y,\z)$ such that
\[
\Ga \mdl{\V} \Pi'' \quad \text{and} \quad \Pi'', \Si \mdl{V} \De.
\]
Hence, using the first equivalence above, $\Pi'' \mdl{\V} \Pi'$, and so also $\Ga \mdl{\V} \Pi'$. Finally then, using the second equivalence above, $\Ga \mdl{\V} \Pi$. We conclude that $\V$ admits left uniform Maehara interpolation. 
\end{proof}


\section{Uniform interpolation and model completions}\label{sec:modelcompletions}

In this section we provide sufficient conditions for the first-order theory of a variety $\V$ to have a model completion. In particular, we show that this is the case if $\V$ admits left and right Maehara uniform interpolation (Corollary~\ref{cor:maeharamodelcomp}). Let us emphasize that a variety of algebras over an algebraic signature $\lang$ may also be understood as a class of structures over the first-order language $\lang$ in the sense of model theory. That is, an algebra is a structure in the first-order language containing a function symbol for each operation symbol of $\lang$. By the first-order theory of a variety, we mean the set of first-order sentences in this language that are true in all algebras of the variety.

We first recall some relevant notions from model theory, referring to~\cite{CK90}*{Sec~3.5} for further details. Two first-order theories $T$ and $T'$ are called \emph{co-theories} if they entail the same universal sentences. Semantically, $T'$ is a co-theory of $T$ if, and only if, every model of $T$ embeds into a model of $T'$ and vice versa. A first-order theory $T^*$ is called \emph{model complete} if every formula is equivalent over $T^*$ to an existential formula. That is, model complete theories are those in which alternations of quantifiers can be eliminated. Semantically, a theory $T^*$ is model complete if, and only if, every embedding between models of $T^*$ is elementary. A theory $T^*$ is called a \emph{model companion} of $T$ if it is a model complete co-theory of $T$. A {\em model completion} of $T$ is a model companion $T^*$ such that for any model $M$ of $T$, the theory of $T^*$ together with the diagram of $M$ is complete. 

In the following proposition, we collect some useful facts related to model completions:

\begin{proposition}[See, e.g.,~\cite{CK90}*{Prop.~3.5.13,~3.5.15,~3.5.18,~3.5.19}]\label{prop:modcomp} \
\begin{enumerate}
\item[\rm (a)] Any theory has at most one model companion.
\item[\rm (b)] If a $\forall\exists$-theory $T$ has a model companion $T^*$, then $T^*$ coincides with the theory of the existentially closed models for $T$.
\item[\rm (c)] If $T^*$ is a model companion of $T$, then $T^*$ is a model completion of $T$ if, and only if, the class of models of $T$ has the amalgamation property. 

\item[\rm (d)] A model completion of a universal theory admits quantifier elimination.
\end{enumerate}
\end{proposition}

The following theorem, although stated and proved here in a different form, is essentially~\cite{GZ97}*{Thm.~1}.

\begin{theorem}\label{thm:modelcompletion}
Suppose that $\V$ has the amalgamation property and admits left and right uniform deductive interpolation, and that $\KCon \m{A}$ is dually Brouwerian for any finitely presented $\m A$ in $\V$. Then the theory of $\V$ has a model completion.
\end{theorem}
\begin{proof}
Fix a countable set of variables $\y$. Denote by $\mathcal{E}$ the set of finite sets of equations with variables in $\y$. 
Let 
\[
J := \{(\Ga,\De_1\dots,\De_n,x) : n \geq 0, \ \Ga,\De_1,\dots,\De_n \in \mathcal{E},\ x \in \y\}.
\]
For any $\Ga \in \mathcal{E}$, write $\phi_\Ga$ for the first-order formula $\bigwedge_{(\a,\b) \in \Ga} (\a \eq \b)$. For any $\x \subseteq \y$ containing the variables that occur in $\Ga$, let $\theta_\Ga(\x)$ be the compact congruence that $\Ga$ generates on $\F(\x)$. Throughout this proof, we will use the fact (cf. Lemma~\ref{lem:algcons}) that, for any algebra $\m A$ and assignment $f \colon \x \to \m A$, we have $\m A, f \models \phi_\Ga$ if, and only if, $\theta_\Ga(\x) \subseteq \ker \hat{f}$, where $\hat{f} \colon \F(\x) \to \m A$ is the unique extension of $f$ to the free algebra $\F(\x)$. For any non-empty finite sequence $(\Ga,\De_1,\dots,\De_n)$ of elements of $\mathcal{E}$, write $\phi_{(\Ga, \De_1, \dots, \De_n)}$ for the formula $\phi_{\Ga} \land \bigwedge_{m=1}^n \neg \phi_{\De_m}$.

Fix $j = (\Gamma,\De_1,\dots,\De_n,x) \in J$. Let $\x$ be the finite set of variables occurring in $\Ga \cup \bigcup_{m=1}^n \De_m$, and let $\z := \x \setminus \{x\}$. Let $i \colon \F(\z) \hookrightarrow \F(\x)$ be the homomorphism induced by the inclusion $\z \subseteq \x$. By the assumption that $\cls{V}$ admits right uniform deductive interpolation and Proposition~\ref{prop:rudipintermediate}, the congruence $i^{-1}(\theta_{\Ga}(\x))$ on $\F(\z)$ is compact, so we may pick a finite set of equations $\Si$ that generates this congruence. For $1 \leq m \leq n$, the congruences $\forall_i(\theta_{\De_m}(\x) - \theta_{\Ga}(\x))$ on $\F(\z)$ are compact by the assumption that $\cls{V}$ admits left uniform deductive interpolation and Proposition~\ref{prop:ludipintermediate}, so we may also pick finite sets of equations $\Pi_m$ that (interpreted in $\F(\z)$) generate these congruences. Now define $\psi_j$ to be the first-order formula $\phi_{(\Si,\Pi_1,\dots,\Pi_n)} \to \exists x \phi_{(\Ga,\De_1,\dots,\De_n)}$.

Let $T$ be the theory of $\V$ and define $T^* := T \cup \{\psi_j : j \in J\}$. We claim that $T^*$ is the model completion of $T$. To prove this, we will show that $T^*$ has quantifier elimination (cf. Proposition~\ref{prop:modcomp}(d)), and that $T^*$ is a co-theory of $T$. The result then follows by Proposition~\ref{prop:modcomp}(c).

\vspace{5mm}

{\bf $T^*$ has quantifier elimination.}
We prove that
\begin{equation}\label{eq:quantelim}
\text{for all } (\Ga,\De_1,\dots,\De_n,x) \in J, \quad T \vdash \left(\exists x \phi_{(\Ga,\De_1,\dots,\De_n)} \right)\to \phi_{(\Si,\Pi_1,\dots,\Pi_n)}.
\end{equation} 
From (\ref{eq:quantelim}), and the fact that $\psi_j \in T^*$ by definition, it follows that in $T^*$ the existential quantifier in any formula of the form $\exists x \phi_{(\Ga,\De_1,\dots,\De_n)}$ can be eliminated. It then follows that $T^*$ has quantifier elimination~\cite{CK90}*{Lemma~1.5.1}.

For the proof of (\ref{eq:quantelim}), let $j = (\Ga,\De_1,\dots,\De_n,x) \in J$.
Let $\m A$ be an algebra in $\V$. For any variable assignment $g \colon \z \to \m A$, we show that $\m A, g \models \exists x \phi_{(\Ga,\De_1,\dots,\De_n)}$ implies $\m A, g \models \phi_{(\Si,\Pi_1,\dots,\Pi_n)}$.
Let $g \colon \z \to \m A$ be any variable assignment such that $\m A, g \models \exists x \phi_{(\Ga,\De_1,\dots,\De_n)}$. We prove that $\m A, g \models \phi_{\Si}$ and $\m A, g \models \neg \phi_{\Pi_m}$ for all $1 \leq m \leq n$. Pick $a \in A$ such that $\m{A}, g \models \phi_{(\Ga,\De_1,\dots,\De_n)}(a)$. 
Extend $g$ to an assignment $f \colon \x \to \m A$ by letting $f(x) := a$. 
Since $\m A, f \models \phi_{\Ga}$, we have $\theta_{\Ga}(\x) \subseteq \ker \hat{f}$.
Hence, $\theta_{\Si}(\z) = i^{-1}(\theta_{\Ga}(\x)) \subseteq i^{-1}(\ker \hat{f}) = \ker \hat{f'}$, since $i \circ \hat{f'} = \hat{f}$, proving $\m A, f' \models \phi_{\Si}$.
Towards a contradiction, suppose that $\m A, f' \models \phi_{\Pi_k}$ for some $1 \leq k \leq n$. Then $i_*(\theta_{\De_k}(\x) - \theta_{\Ga}(\x)) = \theta_{\Pi_k}(\z) \subseteq \ker \hat{f'}$. By adjunction, we get $\theta_{\De_k}(\x) \subseteq \theta_{\Ga} \lor i^*(\ker \hat{f'}) \subseteq \ker \hat{f}$. But this contradicts the fact that $\m A, f \models \neg \phi_{\De_k}$.\\

{\bf $T^*$ is a co-theory of $T$.}
Since $T \subseteq T^*$, it suffices to show that any universal sentence entailed by $T^*$ is entailed by $T$. We will show first that, for any $j = (\Ga,\De_1,\dots,\De_n,x) \in J$, $\m A'$ in $\V$, and assignment $f' \colon \z \to \m A'$, 
\begin{multline}\label{eq:cotheory}
\text{if } \m A', f' \models \phi_{(\Si,\Pi_1,\dots,\Pi_n)}, \text{ then there exists } \iota \colon \m A' \hookrightarrow \m A \text{ such that } \\ \m A, \iota f' \models \exists x \phi_{(\Ga,\De_1,\dots,\De_n)}.
\end{multline}

To prove (\ref{eq:cotheory}), let $\hat{f'} \colon \F(\z) \to \m A'$ be the unique extension of $f'$ to a homomorphism. We first assume that $\hat{f'}$ is surjective, then subsequently deduce the general case of (\ref{eq:cotheory}).

Let $i \colon \F(\z) \into \F(\x)$ be the natural inclusion homomorphism. Consider the congruence $\theta := \theta_{\Ga} \lor i^*(\ker \hat{f'})$ on $\F(\x)$. Let $\m A$ be the quotient of $\F(\x)$ by $\theta$, and denote the natural quotient map by $p \colon \F(\x) \twoheadrightarrow \m A$. Note that, using the assumption $\m A', f' \models \phi_{\Si}$, we have $\theta_{\Si} = i^{-1}(\theta_{\Ga}) \subseteq \ker \hat{f'}$, and also $i^{-1}i^*(\ker \hat{f'}) = \ker \hat{f'}$ by~\cite{MMT14}*{Lemma~5}. It follows that $i^{-1}(\theta) = \ker \hat{f'}$, so there is a well-defined injective homomorphism $\iota \colon \m A' \hookrightarrow \m A$ which satisfies $\iota \circ \hat{f'} = p \circ i$. 

\begin{center}
\begin{tikzpicture}
\matrix (m) [matrix of math nodes, row sep=1.5em, column sep=1.5em]{
\F(\z)	& & \F(\x) \\
&	& \\
\m A'	& & \m A &\\
};

\path
(m-1-1) edge[right hook->] node[above] {$i$} (m-1-3)
(m-1-1) edge[->>] node[left] {$\hat{f'}$} (m-3-1)
(m-1-3) edge[->>] node[right] {$p$} (m-3-3)
(m-3-1) edge[right hook->] node[above] {$\iota$} (m-3-3);

\end{tikzpicture}
\end{center}
Let $f \colon \x \to \m A$ be the extension of the assignment $\iota f'$ by setting $f(x) := p(x)$. Note that $\hat{f} = p$, since these two homomorphisms agree on all the variables.
To conclude the proof of (\ref{eq:cotheory}), we will show that $\m A, f \models \phi_{(\Ga,\De_1,\dots,\De_n)}$. By definition, $\theta_{\Ga} \subseteq \theta = \ker p = \ker \hat{f}$. Also, for each $1 \leq k \leq n$, we have $\m A',f' \models \neg \phi_{\Pi_k}$, so $i_*(\theta_{\De_k} - \theta_{\Ga}) = \theta_{\Pi_k} \not\subseteq \ker \hat{f'}$, so by adjunction $\theta_{\De_k} \not\subseteq \theta_{\Ga} \lor i^{*}(\ker \hat{f'}) = \theta$, as required.

%
For the general case of (\ref{eq:cotheory}), let $\m B'$ be the image of $\hat{f'}$ in $\m A'$, and find an extension $\iota' \colon \m B' \to \m B$ with $\m B, \iota'f' \models \exists x \phi_{(\Ga,\De_1,\dots,\De_n)}$, witnessed by $b \in B$, say. Since $\V$ has the amalgamation property, applying this to the injections $\iota'$ and the inclusion $\m B' \to \m A'$, we find an extension $\iota \colon \m A' \to \m A$ together with a map $\lambda \colon \m B \to \m A$ such that $\lambda \iota' = \iota|_{\m B'}$. Since $\phi_{(\Ga,\De_1,\dots,\De_n)}$ is quantifier-free, we will now have $\m A, \iota f \models \exists x \phi_{(\Ga,\De_1,\dots,\De_n)}$ witnessed by $x := \lambda(b)$.

To finish the proof, let $\chi$ be a universal sentence such that $T^* \vdash \chi$. By the compactness theorem of first-order logic, there exists a finite subset $F \subseteq J$ such that $T \cup \{\psi_j : j \in F\} \vdash \chi$. We prove that, in fact, $T \vdash \chi$. Let $\m A'$ be any algebra in $\V$. Let $\z$ be the set of free variables occurring in $\{\psi_j : j \in F\}$, and pick an arbitrary assignment $f' \colon \z \to \m A'$. By repeatedly applying (\ref{eq:cotheory}), we obtain an extension $\iota \colon \m A' \into \m A$ such that $\m A, \iota f' \models \psi_j$ for each $j \in F$. Now, since $T \cup \{\psi_j : j \in F\} \vdash \chi$, we have $\m A\models \chi$. Since $\m A'$ is a subalgebra of $\m A$ and $\chi$ is universal, it follows that $\m A'\models \chi$, as required.
\end{proof}

\begin{corollary}
Suppose that $\V$ is locally finite and congruence distributive, has the amalgamation property, and for any finite sets $\x,\y$, the compact lifting of the inclusion homomorphism $i \colon \F(\y) \into \F(\x,\y)$ preserves intersections. Then the theory of $\V$ has a model completion.
\end{corollary}

\begin{example}[Sugihara monoids]
The variety of Sugihara monoids is locally finite, congruence distributive (as it has a lattice reduct), and has the amalgamation property~\cite{MM12}. Moreover, for any finite sets $\x,\y$, the compact lifting of the inclusion homomorphism $i \colon \F(\y) \into \F(\x,\y)$ preserves intersections. Hence the theory of Sugihara monoids admits a model completion.
\end{example}

\begin{corollary}\label{cor:maeharamodelcomp}
If a variety $\V$ admits both left and right uniform Maehara interpolation, then the theory of $\V$ has a model completion.
\end{corollary}

\begin{example}[Varieties of Heyting algebras]
The non-trivial varieties of Heyting algebras whose theories have a model completion are exactly those that admit left and right uniform interpolation, which, by the results in \cite{GZ02}*{Ch.~4}, are exactly the seven varieties identified by Maksimova that admit deductive interpolation (equivalently, have the amalgamation property).
\end{example}

Let us conclude by comparing our Theorem~\ref{thm:modelcompletion} to \cite{GZ02}*{Thm.~3.8}, which states that if the opposite category of the finitely presented algebras in a variety is an r-Heyting category, then the variety has a model completion.

We briefly recall the definition of a r-Heyting category from \cite{GZ02}*{Ch.~3}, where the reader can find more details. Note that the following category-theoretic definitions will be applied to the {\em opposite} of a category of finitely presented algebras. In particular, `regular subobjects' correspond to compact congruences, `epis' are subalgebra inclusions, `pullback functors' are the {\em covariant} liftings of the form $h^*$ (Def.~\ref{dfn:lifting}), and left and right adjoints are swapped with respect to the rest of this paper, also see Remark~\ref{rem:catthry} above. An {\em r-regular category} is defined to be a category which has finite limits, each arrow factorizes as an epi followed by a regular mono, and epis are stable under pullbacks. By \cite{GZ02}*{Prop.~3.3}, a category with finite limits is r-regular if, and only if, the pullback functors operating on regular subobjects have left adjoints satisfying the Beck-Chevalley condition. An {\em r-Heyting category} is an r-regular category in which finite joins of regular subobjects exist, and the pullback functors on regular subobjects have right adjoints.
\begin{proposition}\label{prop:comparison}
Let $\V$ be a variety with the congruence extension property. The following are equivalent:
\begin{enumerate}
\item The opposite of the category of finitely presented algebras in $\V$ is an r-Heyting category. 
\item $\V$ admits left and right uniform deductive interpolation, and $\KCon \m{A}$ is a Heyting algebra for any finitely presented $\m{A}$ in $\V$.
\item  $\V$ has the left and right Maehara uniform interpolation properties, and $\KCon \m{A}$ has finite meets for any finitely presented $\m{A}$ in $\V$.
\end{enumerate}
\end{proposition}
\begin{proof}
(i) $\Leftrightarrow$ (ii). By \cite{GZ02}*{Prop.~3.3} cited above, $\V$ is r-regular if, and only if, $\V$ admits right uniform deductive interpolation, since the Beck-Chevalley condition for left adjoints is equivalent to having the amalgamation property in the presence of the congruence extension property (see, e.g., \cites{GRS03,Pit83}). The equivalence now follows from Theorem~\ref{thm:ludipalgebraic} and the fact that Heyting algebras are exactly dually Brouwerian semilattices in which all finite meets exist. (Recall that finite joins of regular subobjects correspond to finite meets of compact congruences.)

(ii) $\Leftrightarrow$ (iii). By Theorem~\ref{thm:ludipalgebraic}, the conjunction of left uniform deductive interpolation and $\KCon \m{A}$ being dually Brouwerian for all finitely presented $\m{A}$ is equivalent to deductive interpolation and all compact liftings having left adjoints. By Theorems~\ref{thm:maehara} and \ref{thm:umaeharaleft}, in the presence of the congrue=nce extension property this is equivalent to left Maehara uniform interpolation. By Theorem~\ref{thm:umaehararight}, the congruence extension property and right uniform deductive interpolation are together equivalent to right Maehara uniform interpolation.
\end{proof}

\begin{remark}
In light of Proposition~\ref{prop:comparison}, our result in Theorem~\ref{thm:modelcompletion} is slightly more general than \cite{GZ02}*{Thm.~3.8}, in that we do not need to assume the congruence extension property and that finite meets of compact congruences exist. These properties were added in \cite{GZ02} to obtain the partial converse \cite{GZ02}*{Thm.~3.11}, which says that under certain conditions the opposite category of finitely presented algebras in a variety admitting a model-completion must be r-Heyting. However, close inspection of the proof of \cite{GZ02}*{Thm.~3.8}  shows that the additional assumptions on the category are not needed for the direction which establishes existence of a model completion on the basis of uniform interpolation properties. The slight generalization we obtain here may be useful for applications in contexts where the congruence extension property or the existence of finite meets of compact congruences fails.
\end{remark}


\bibliographystyle{model1a-num-names}

\end{document}